\documentclass[preprint,12pt]{elsarticle}

\usepackage[papersize={210mm,297mm},lmargin=2.5cm,rmargin=2.5cm,top=3cm,bottom=3cm]{geometry}

\usepackage{amsmath,amsfonts,amssymb,amsthm}
\usepackage{latexsym,pdfsync,xcolor,graphicx}
\usepackage[nottoc]{tocbibind}
\usepackage[T1]{fontenc}
\usepackage[utf8]{inputenc}												
\usepackage[english]{babel}			
\usepackage{tikz}
\usepackage{appendix}
\usepackage{hyperref}
\usepackage{enumerate}
\usepackage[shortlabels]{enumitem}
\usepackage{longtable}
\usepackage{float}
\usepackage{multirow,array}
\usepackage{verbatim}
\usepackage{titlesec}
\usepackage{graphicx}
\usepackage{caption}
\usepackage{subcaption}
\usepackage{multicol}

\newtheorem{theorem}{Theorem}[section]

\newtheorem{proposition}[theorem]{Proposition}
\newtheorem{remark}[theorem]{Remark}
\numberwithin{equation}{section}
\newtheorem*{conjecture*}{Conjecture}

\newcommand{\R}{\ensuremath{\mathbb{R}}}

\journal{****}

\begin{document}

\begin{frontmatter}

\title{Global phase portraits of a predator-prey system}
\author[a]{\'Erika Diz-Pita}
\address[a]{Departamento de Estat\'istica, An\'alise Matem\'atica e Optimizaci\'on, Universidade
	de Santiago de Compostela, 15782 Santiago de Compostela, Spain }
\ead{erikadiz.pita@usc.es}
\author[b]{Jaume Llibre}
\address[b]{Departament de Matem\`atiques, Universitat Aut\`onoma de Barcelona, 08193 Bellaterra, Barcelona, Spain}
\ead{jllibre@mat.uab.cat}
\author[a]{M. Victoria Otero-Espinar}
\ead{mvictoria.otero@usc.es}
\begin{abstract}
	
We classify the global dynamics of a family of Kolmogorov systems depending on three parameters which has ecological meaning as it modelizes a predator-prey system. We obtain all their topologically distinct global phase portraits in the positive quadrant of the Poincaré disc, so we provide all the possible distinct dynamics of these systems.
\end{abstract}
\begin{keyword}
Predator-prey system \sep Kolmogorov system \sep global phase portrait \sep Poincar\'e disc.
\end{keyword}

\end{frontmatter}
\section{Introduction} 

Rosenzweig and MacArthur introduced in \cite{RM} the following predator-prey model
\begin{equation*}
	\begin{split}
		\dot{x}&=r x \left( 1- \frac{x}{K} \right) - y \frac{m x }{b+x}, \\[0.1cm]
		\dot{y}&= y \left(- \delta + c \frac{m x }{b+x} \right),
	\end{split}
\end{equation*}
where the dot as usual denotes derivative with respect to the time $t$, $x\ge 0$ denotes the prey density (\#/unit of area) and $y\ge 0$ denotes the predator density (\#/unit of area), the parameter $\delta>0$ is the death rate of the predator, the function $mx/(b+x)$ is the \# prey caught per predator per unit time, the function $x\to rx ( 1-x/K)$ is the growth of the prey in the absence of predator, and $c>0$ is the rate of conversion of prey to predator. 

The Rosenzweig and MacArthur system is a particular system of the general predator–prey systems with a Holling type II, see \cite{H1, H2}.

In \cite{Huzak} Huzak reduced the study of the Rosenzweig and MacArthur system to study a polynomial differential system. In order to do that the first step is to do the rescaling $(\overline{x},\overline{y},\overline{b}, \overline{c},\overline{\delta})=(x/K,(m/rK)y,b/K, cm/r,\delta/r)$. After denoting again $(\overline{x},\overline{y},\overline{b}, \overline{c},\overline{\delta})$ by $(x,y,b,c,\delta)$ and doing a time rescaling multiplying by $b+x$, the obtained polynomial differential system of degree three is
\begin{equation}\label{sis}
	\begin{split}
		\dot{x}&=x(-x^2+(1-b)x-y+b),\\
		\dot{y}&=y((c-\delta)x - \delta b),
	\end{split}
\end{equation}
where $b,c$ and $\delta$ are positive parameters. This system is studied in the positive quadrant of the plane $\mathbb{R}^2$ where it has ecological meaning. See systems (1.1) and (2.2) of \cite{Huzak}.

Huzak \cite{Huzak}  focuses his work  in the study of the periodic sets that can produce the canard relaxation oscillations after perturbations. He finds three types of limit periodic sets and studies their cyclicity by using the geometric singular perturbation theory and the family blow-up at $(x,y,\delta)=(0,br/m,0)$, where $\delta$ is the singular perturbation parameter. He proves that the upper bound on the number of limit cycles of the system is 1 or 2 depending on the parameters. 

Systems \eqref{sis} are particular Kolmogorov systems. These systems were proposed in 1936, see \cite{Kolmogorov}, as an extension of the Lotka-Volterra systems to arbitrary dimension and arbitrary degree. 

We want to complete the study of the dynamics of systems \eqref{sis} and classify all their phase portraits on the closed positive quadrant of the Poincaré disc, in this way we also can control the dynamics of the system near the infinity. This classification is given in the following result, except for the case with the parameters satisfying $0<b\delta<c-\delta$, $\delta(\delta(b+1)+c(b-1))^2 - 4 c (c-\delta)^2 (c-\delta(b-1))<0$ and $1+c-\delta-b-b\delta>0$, in which we make a conjecture about the expected global phase portrait.

\begin{theorem}
	The global phase portait of system \eqref{sis} in the closed positive quadrant of the Poincaré disc is topologically equivalent to one of the $3$ phase portraits of Figure \ref{fig:globales} in the following way:
	\begin{itemize}
		\item If $b\delta\geq c-\delta$ the phase portrait is equivalent to phase portrait (A). \vspace{0.1cm}
		
		\item If $0<b\delta<c-\delta$ and $\delta(\delta(b+1)+c(b-1))^2 - 4 c (c-\delta)^2 (c-\delta(b-1))\geq0$ the phase portrait is equivalent to phase portrait (B). 
		\vspace{0.1cm}
		
		\item If $0<b\delta<c-\delta$ and $\delta(\delta(b+1)+c(b-1))^2 - 4 c (c-\delta)^2 (c-\delta(b-1))<0$ and $1+c-\delta-b-b\delta<0$ the phase portrait is equivalent to phase portrait (C). \vspace{0.1cm}
		
	\end{itemize}
	
\end{theorem}

\begin{figure}[H]
	\centering
	\begin{subfigure}[h]{3.7cm}
		\centering
		\includegraphics[width=3.5cm]{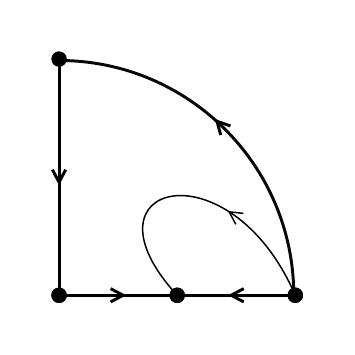}
		\caption*{(A)}
	\end{subfigure}
	\begin{subfigure}[h]{3.7cm}
		\centering
		\includegraphics[width=3.5cm]{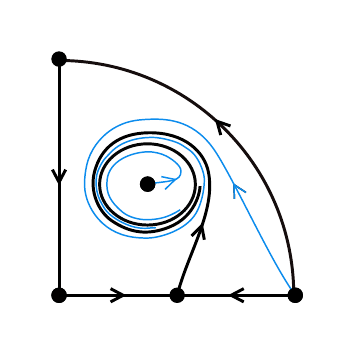}
		\caption*{(B)}
	\end{subfigure}
	\begin{subfigure}[h]{3.7cm}
		\centering
		\includegraphics[width=3.5cm]{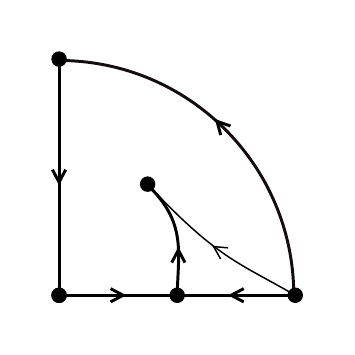}
		\caption*{(C)}
	\end{subfigure}
	\caption{Phase portraits of system \eqref{sis} in the positive quadrant of the Poincaré disc.}
	\label{fig:globales}
\end{figure}

\begin{conjecture*}
	The global phase portait of system \eqref{sis} in the closed positive quadrant of the Poincaré disc if $0<b\delta<c-\delta$ and $\delta(\delta(b+1)+c(b-1))^2 - 4 c (c-\delta)^2 (c-\delta(b-1))<0$  and $1+c-\delta-b-b\delta>0$ is also topologically equivalent to the one in Figure \ref{fig:globales}(C).
\end{conjecture*}

In Figure \ref{fig:regiones} are represented the regions and surfaces in the parameters space in which each one of the phase portraits are realised. In the region I and over the surface $S_1$ the phase portrait is the one in Figure \ref{fig:globales}(A) and in the region III the phase portrait is the one in Figure \ref{fig:globales}(B). In region II there are two subregions, II-a and II-b. It is proved that in the region II-a the phase portrait is the one in Figure \ref{fig:globales}(C) and we conjecture that the phase portrait is the same in the region II-b and over the surfaces $S_2$ and $S_3$.

\begin{figure}
	\centering
	\captionsetup{width=\textwidth}
	\includegraphics[width=8cm]{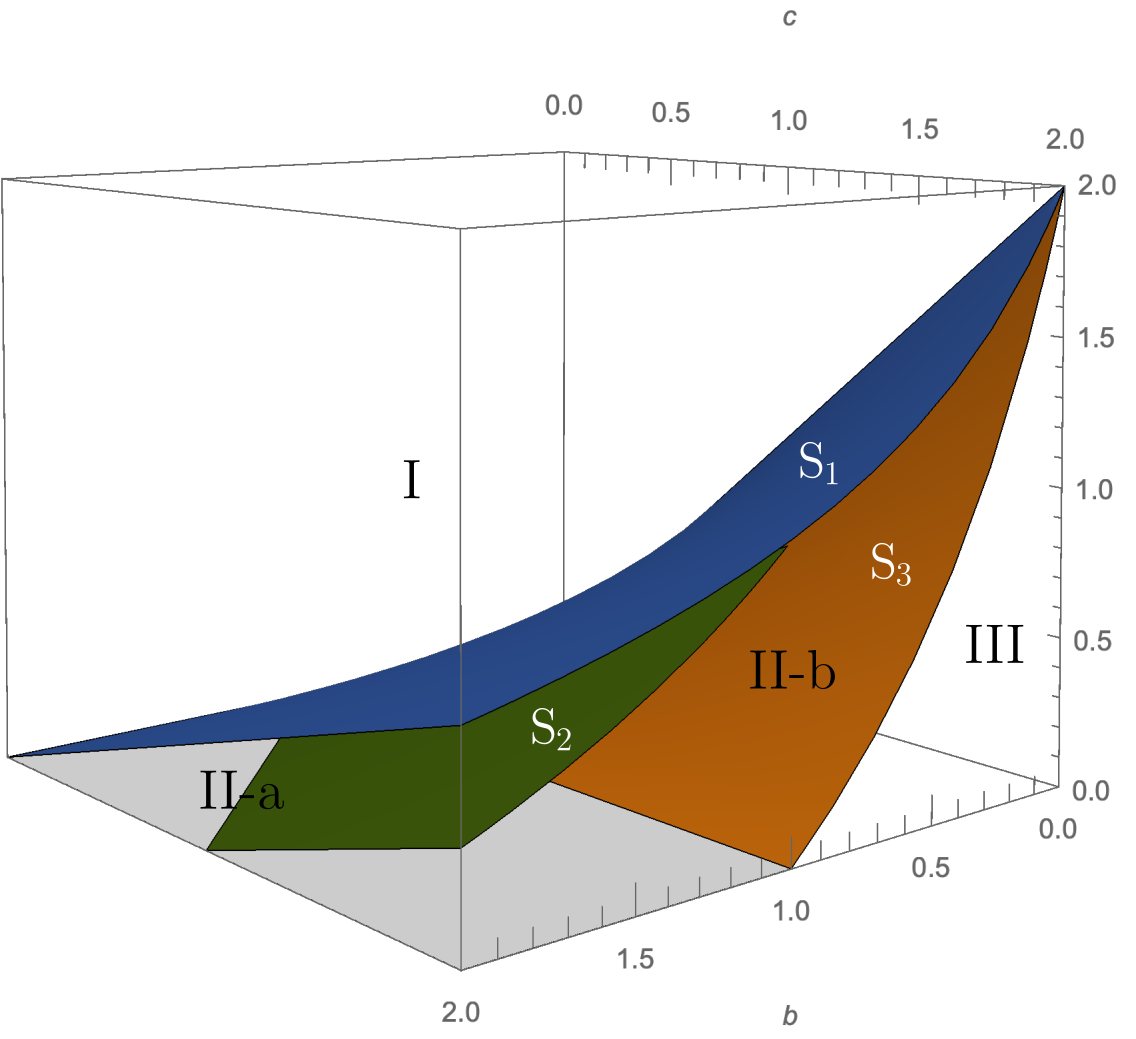}
	\caption{The regions I, II-a, II-b, III and the surfaces separating the different phase portraits: $S_1:\left\lbrace \delta=c/(b+1) \mid b, c \geq 0 \right\rbrace $, $S_2:\left\lbrace \delta= (1+c-b)/(b+1) \mid b,c \geq 0, (1+c-b)/(b+1)< c/(b+1) \right\rbrace $ and $S_3:\left\lbrace \delta= c(1-b)/(1+b) \mid b,c \geq 0 \right\rbrace $ }
	\label{fig:regiones}.
\end{figure}

\section{Preliminaries}\label{sec:preliminaries}

Here we introduce the Poincar\'e compactification, as it allows to control the dynamics of a polynomial differential system near the infinity.

Consider a polynomial system in  $\mathbb{R}^2$
\begin{equation*}
	\begin{split}
		\dot{x_1}&=P(x_1,x_2), \\
		\dot{x_2}&=Q(x_1,x_2),
	\end{split}
\end{equation*}
of degree $d$; the sphere $\mathbb{S}^2=\left\lbrace  y \in \mathbb{R}^3 : y_1^2+y_2^2+y_3^2=1 \right\rbrace$, which we will call 
the \textit{Poincar\'e sphere}, and its tangent plane at the point $(0,0,1)$ which we identify with $\mathbb{R}^2$.

We can obtain an induced vector field in $\mathbb{S}^2 \textbackslash \mathbb{S}^1$
by means of central projections $f^+:\mathbb{R}^2\to \mathbb{S}^2$ and $f^-:\mathbb{R}^2\to \mathbb{S}^2$, which are defined as 
\begin{equation*}
	f^+(x)=\left( \dfrac{x_1}{\Delta(x)}, \dfrac{x_2}{\Delta(x)}, \dfrac{1}{\Delta(x)} \right) \;\; \text{and} \;\; f^-(x)=\left( \dfrac{-x_1}{\Delta(x)}, \dfrac{-x_2}{\Delta(x)}, \dfrac{-1}{\Delta(x)} \right), 
\end{equation*}
where $\Delta(x)=\sqrt{x_1^2+x_2^2+1}$.
The differential $Df^+$ and $Df^-$ provide a vector field in the northern and southern hemisphere respectively. The points of the equator $\mathbb{S}^1$ of $\mathbb{S}^2$ correspond with 
the points at infinity of $\mathbb{R}^2$, and we can extend analytically the vector field to these points of the equator  multiplying the field by $y_3^d$. This extended field is called the \textit{Poincar\'e compactification} of the original vector field. Then  we must study the dynamics of the Poincaré compactification near $\mathbb{S}^1$, for studying the dynamics of the original field in the neighborhood of the infinity.

We will work in the local charts  $(U_i,\phi_i)$ and $(V_i,\psi_i)$ of the sphere $\mathbb{S}^2$, where $U_i= \left\lbrace y \in \mathbb{S}^2 :y_i > 0  \right\rbrace$, $V_i = \left\lbrace y \in \mathbb{S}^2 : y_i < 0  \right\rbrace$, $\phi_i: U_i \longrightarrow \mathbb{R}^2$ and $\psi_i: V_i \longrightarrow \mathbb{R}^2$ for $i=1,2,3$ with $\phi_i(y) = \psi_i(y) = \left( y_m/y_i , y_n/y_i \right)$ for $m < n$ and $m,n \neq i$.

The expression of the Poincaré compactification in the local chart $(U_1,\phi_1)$ is
\begin{equation}\label{Poincare_comp_U1}
	\dot{u}= v^d \left[ -u \: P\left( \frac{1}{v}, \frac{u}{v}\right)  + Q\left( \frac{1}{v}, \frac{u}{v}\right) \right], \;
	\dot{v}= - v^{d+1} \: P\left( \frac{1}{v}, \frac{u}{v}\right),
\end{equation}
in the local chart $(U_2,\phi_2)$ is
\begin{equation}\label{Poincare_comp_U2}
	\dot{u}= v^d \left[ P\left( \frac{u}{v}, \frac{1}{v}\right) - u Q\left( \frac{u}{v}, \frac{1}{v}\right) \right], \;
	\dot{v}= - v^{d+1} \: Q\left( \frac{u}{v}, \frac{1}{v}\right),
\end{equation}
and in the local chart $(U_3,\phi_3)$ the expression is
\begin{equation}\label{Poincare_comp_U3}
	\dot{u}= P(u,v), \;
	\dot{v}= Q(u,v).
\end{equation}

The expression for the Poincaré compactification in the local charts $(V_i,\psi_i)$, with $i=1,2,3$ is the same as in the charts $(U_i,\phi_i)$ multiplied by $(-1)^{d-1}$.

As we want to study the behaviour near the infinity, we must study the \textit{infinite singular points}, i.e., the singular points of the Poincaré compactification which lie in the equator $\mathbb{S}^1$. Note that it will be enough to study the infinite points on the local chart $U_1$ and the origin of the local chart $U_2$, because if $y\in \mathbb{S}^1$ is an infinite singular point, then $-y$ is also an infinite singular point and they have the same or opposite stability depending on whether the system has odd or even degree.

We shall present the phase portraits of the polynomial differential systems \eqref{sis} in the \textit{Poincar\'e disc}, i.e. the orthogonal projection of the closed northern hemisphere of $\mathbb{S}^2$ onto the plane $y_3=0$. This will be enough since the orbits of the Poincaré compactification  on $\mathbb{S}^2$ are symmetric with respect to the origin of $\mathbb{R}^3$ so we only need to consider the flow in the closed northern hemisphere.

See chapter 5 of \cite{Libro}  for more details about the Poincaré compactification.

\section{Finite Singular Points}\label{sec:finite}

First we study the finite singular points of system \eqref{sis} in the closed positive quadrant. The origin $P_0=(0,0)$ and the point $P_1=(1,0)$ are singular points for any values of the parameters, and $P_2=\left( b \delta /(c-\delta), (-b c (\delta +b \delta - c))/(c-\delta)^2\right)$ is a positive singular point if $c\neq \delta$ and $0<b\delta<c-\delta$. Note that if $b\delta=c-\delta$ then $P_1= P_2$.

Now we study the local phase portraits at these singular points. The origin is a saddle point, as the eigenvalues of the Jacobian matrix at this point are $b$ and $-\delta b$. At the point $P_1$ the eigenvalues are $-b-1$ and $-\delta b + c - \delta$. The first eigenvalue is always negative, but we distinguish three cases depending on the second one. If $c-\delta < b \delta$ then $P_1$ is a stable node; if $c-\delta >b \delta $ then $P_1$ is a saddle (this was the case in \cite{Huzak} because there $\delta>0$ was kept very small). If $c-\delta= b \delta $, then $P_1$ is a semi-hyperbolic singular point, so from \cite[Theorem 2.19]{Libro} we obtain that $P_1=P_2$ is a saddle-node. 

At the singular point $P_2$ the eigenvalues of the Jacobian matrix are 
\begin{equation*}
	\lambda_{1,2}=\frac{2}{(c-\delta)^2}(A\pm \sqrt{\delta B}), 
\end{equation*}
where 
\begin{equation*}
	A=\delta (c-\delta) - b \delta (c + \delta) \;\; \text{and} \;\; B=\delta(\delta(b+1)+c(b-1))^2 - 4 c (c-\delta)^2 (c-\delta(b-1)).
\end{equation*}

If $B<0$ then the eigenvalues are complex. In this case for $A>0$ the singular point $P_2$ is an unstable focus, and for $A<0$ it is a stable focus. We deal with this case $B<0$ in Section \ref{sec:cases3-7}, where we study the Hopf bifurcation which takes place at $P_2$. 

If $B=0$ we have $\lambda_1=\lambda_2= A/(c-\delta)^2$ and in this case $A$ cannot be zero, because if $A=0$ then $b=(c-\delta)/(c+\delta)$, and replacing this expression $B= - 4 c^2 (c-\delta)^3 /(c+\delta)$, so one of the conditions $c=0$ or $c-\delta=0$ must hold, but this is a contradiction as $c>0$ from the hypotheses, and if $c=\delta$ then $b=0$ again in contradiction with the hypotheses. Then $A\neq0$ and its sign determines if the singular point is either a stable or an unstable node. 

If $B>0$ both eigenvalues are real. The determinant of the Jacobian matrix is 
\begin{equation*}
	- \dfrac{b^2 c \delta}{(c-\delta)^2}(b\delta + \delta - c), 
\end{equation*}
which is positive because the singular point $P_2$ exists only if condition $b\delta<c-\delta$ holds. Then both eigenvalues are nonzero and have the same sign, particularly, if $A>0$ both are positive and $P_2$ is an unstable node, and if $A<0$ both are negative and $P_2$ is a stable node. 

The local phase portrait of the singular point $P_2$ in the case with $A=0$ will be proved in Subsection \ref{subsec:limitcyc}.

In summary, we describe in Table \ref{tab:cases} the finite singular points according the values of the parameters $b$, $c$ and $\delta$.

\begin{table}[H]
	\begin{center}
		\begin{tabular}{|cll|}
			\hline
			\textbf{Case} & \textbf{Conditions} & \textbf{Finite singular points} \\
			\hline
			\hline
			1&  $b\delta >c-\delta$. & $P_0$ saddle, $P_1$ stable node. \\	\hline
			2&  $b\delta =c-\delta$. & $P_0$ saddle, $P_1$ saddle-node.\\	\hline
			3&  $0<b\delta <c-\delta$, $B\geq0$, $A>0$. & $P_0$ saddle, $P_1$ saddle, $P_2$ unstable node. \\	\hline
			4&  $0<b\delta <c-\delta$, $B\geq0$, $A<0$. & $P_0$ saddle, $P_1$ saddle, $P_2$ stable node. \\	\hline
			5&  $0<b\delta <c-\delta$, $B<0$, $A>0$. & $P_0$ saddle, $P_1$ saddle, $P_2$ unstable focus. \\	\hline
			6&  $0<b\delta <c-\delta$, $B<0$, $A<0$. & $P_0$ saddle, $P_1$ saddle, $P_2$ stable focus. \\	\hline
			7&  $0<b\delta <c-\delta$, $B<0$, $A=0$. & $P_0$ saddle, $P_1$ saddle, $P_2$ weak stable focus. \\	\hline
		\end{tabular}
		\caption{The finite singular points in the closed positive quadrant.} \label{tab:cases}
	\end{center}
\end{table}

\section{Infinite Singular Points}\label{sec:infinite}

In this section we will consider the Poincar\'e compactification of system \eqref{sis} as it allows to study the behavior of the trajectories near infinity. 

In the chart $U_1$ system \eqref{sis} writes
\begin{equation}\label{sisU1}
	\begin{split}
		\dot{u}&=uv^2 - b(\delta+1)uv^2+(b+c-\delta-1)uv+u ,\\
		\dot{v}&=uv^2-bv^3+(b-1)v^2+v.
	\end{split}
\end{equation}

The only singular point over $v=0$ is the origin of $U_1$, which we denote by $O_1$. The linear part of system \eqref{sisU1} at the origin is the identity matrix, so $O_1$ is an unstable node.

In the chart $U_2$ system \eqref{sis} writes
\begin{equation}\label{sisU2}
	\begin{split}
		\dot{u}&=-u^3 + ( \delta + 1 - b - c)u^2 v + b(\delta+1) uv^2 - uv,\\
		\dot{v}&= (\delta - c) u v^2 + b\delta v^3.
	\end{split}
\end{equation}
The origin of $U_2$ is a singular point, $O_2$, and the linear part of system \eqref{sisU2} at $O_2$ is identically zero, so we must use the blow-up technique to study it. We do a horizontal blow up introducing the new variable $w_1$ by means of the variable change $vw_1=u$, and get the system
\begin{equation}\label{sis_blowup2}
	\begin{split}
		&\dot{w_1}=v^2w_1^3 + (1-b) v^2w_1^2 + b w_1 v^2 - w_1 v,\\
		&\dot{v}=(\delta-c)w_1v^3+b\delta v^3.
	\end{split}
\end{equation}
Now rescaling the time variable we cancel the common factor $v$, getting the system
\begin{equation}\label{sis_blowup3}
	\begin{split}
		&\dot{w_1}=vw_1^3 + (1-b) vw_1^2 + b w_1 v - w_1,\\
		&\dot{v}=(\delta-c)w_1v^2+b\delta v^2.
	\end{split}
\end{equation}
The only singular point on $v=0$ is the origin, which is  semi-hyperbolic. Applying \cite[Theorem 2.19]{Libro} we conclude that it is a saddle-node. Studying the sense of the flow over the axis we determine that the phase portrait around the origin of system \eqref{sis_blowup3} is the one on Figure \ref{fig:blowupU2}(a). If we multiply by $v$ the sense of the orbits on the third and fourth quadrants changes and all the points of the $w_1$-axis become singular points. With these modifications we obtain the phase portait for system \eqref{sis_blowup2}, given in Figure \ref{fig:blowupU2}(b). Then we undo the blow up going back to the $(u,v)$-plane. We must swap the third and fourth quadrants and shrink the exceptional divisor to the origin. The phase portrait obtained for system \eqref{sisU2} is not totally determined in the shaded regions of the third and fourth quadrants, see Figure \ref{fig:blowupU2}(c). This can be solved by doing a vertical blow up but, in our case, it is not necessary because we only need to know the phase portrait of $O_2$ in the positive quadrant of the Poincaré disc, which corresponds with the positive quadrant in the plane $(u,v)$, in which the phase portrait is well determined. 

\begin{figure}[H]
	\centering
	\begin{subfigure}[h]{4cm}
		\centering
		\includegraphics[width=3.5cm]{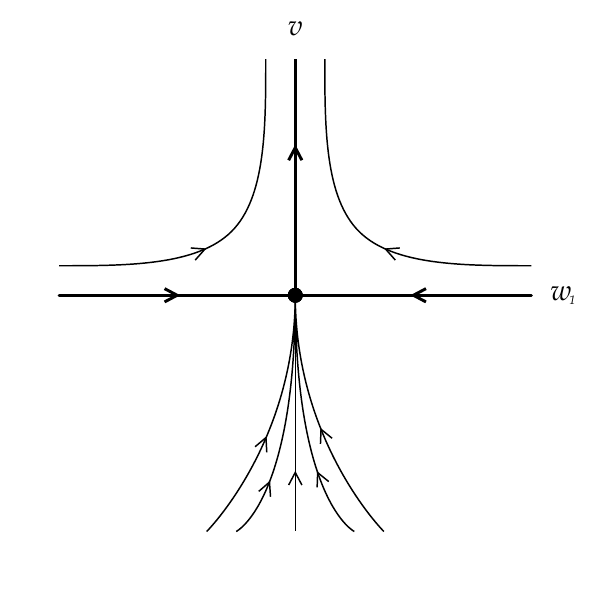}
		\caption*{(a)  Local phase portrait at the origin of system \eqref{sis_blowup3}}
	\end{subfigure}
	\begin{subfigure}[h]{4cm}
		\centering
		\includegraphics[width=3.5cm]{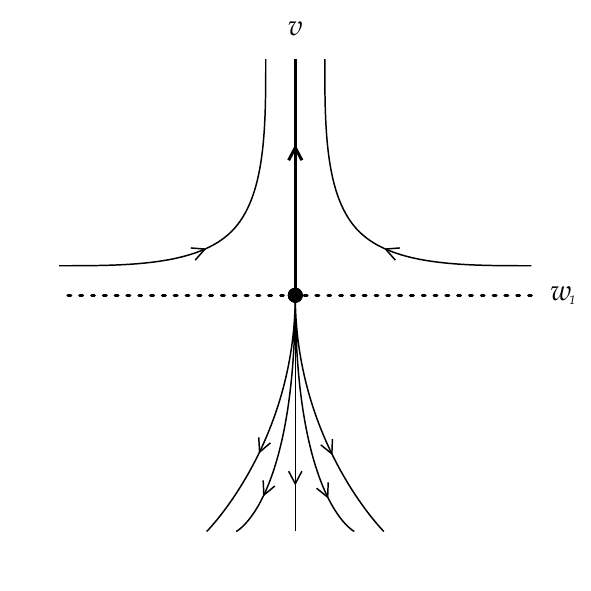}
		\caption*{(b)  Local phase portrait at the origin of system\eqref{sis_blowup2}}
	\end{subfigure}
	\begin{subfigure}[h]{4cm}
		\centering
		\includegraphics[width=3.5cm]{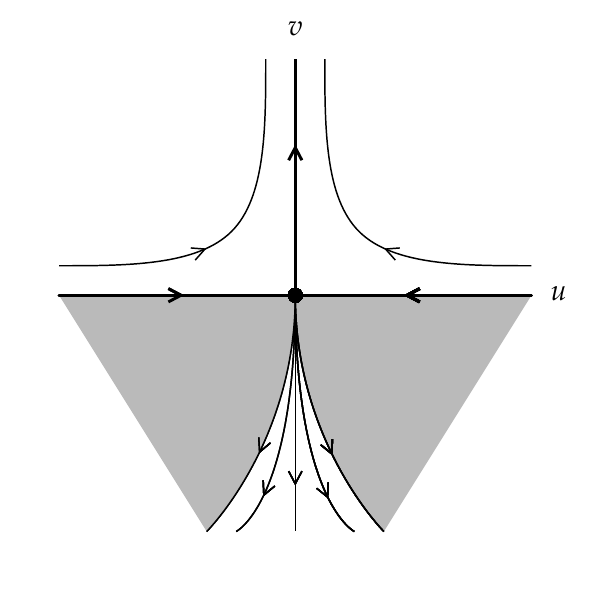}
		\caption*{(c)  Local phase portrait at the origin of system \eqref{sisU2}}
	\end{subfigure}
	\caption{Desingularization of the origin of system \eqref{sisU2}.}
	\label{fig:blowupU2}
\end{figure}

As a conclusion the local phase portrait at the infinite singular points is the same independently of the values of the parameters, so in all cases of Table \ref{tab:cases} the origin of the chart $U_1$, i.e. the singular point $O_1$, is an unstable node and the origin of the chart $U_2$, i.e. the singular point $O_2$  has only one hyperbolic sector on the positive quadrant of the Poincaré disc being one separatrix at infinity and the other on $x=0$.

\section{Cases with no singular points in the positive quadrant}\label{sec:cases12}

In the two first cases of Table \ref{tab:cases} there is no singular points in the positive quadrant. The finite singular points are the origin $P_0$ and $P_1$ which are both over the axes. The axes are invariant lines so there cannot exist a limit cycle surrounding these singular points. Therefore as we have determined the local phase portrait at the finite and infinite singularities, and we know there are no limit cycles, we can study the global portrait in the first quadrant of the Poincaré disc.

In both cases we obtain the same result since in the case in which $P_1$ is a saddle-node, studying the sense of the flow we determine that the parabolic sector of the saddle-node is always on the positive quadrant of the Poincaré disc. Analysing all the possible alpha and omega-limits, the only possibility is that all the orbits leave the infinite singular point $O_1$ and go to the finite singular point $P_1$. This phase portait is given in Figure \ref{fig:globales} (A).

\section{Cases with singular points in the positive quadrant}\label{sec:cases3-7}

\subsection{Existence of limit cycles}\label{subsec:limitcyc}

\begin{theorem}\label{th_LimCic3-5}
	If $0<b\delta<c-\delta$ and $A>0$, then there exists at least one limit cycle surrounding singular point $P_2$. 
\end{theorem}
\begin{proof}
	If conditions $0<b\delta<c-\delta$ and $A>0$ hold, then we have case 3 or 5 of Table \ref{tab:cases}. In both cases singular point $P_1$ is a saddle which has an unstable separatrix on the positive quadrant,  $P_2$ is either an unstable node or an unstable focus, and $O_1$ is an unstable node. By Poincaré-Bendixon theorem, there must exists at least one limit cycle which is the $\omega$-limit of the orbits leaving $O_1$, the orbits leaving $P_2$ and the separatrix of $P_1$, as there are no other singular points that can be the $\omega$-limit of all these orbits.
\end{proof}


In cases 5, 6 and 7 of Table \ref{tab:cases} the Jacobian matrix at the point $P_2$ has complex eigenvalues because $B<0$.  In these cases we study the existence of Hopf Bifurcation, leading to the following result.

\begin{theorem}\label{th_HopfBif}
	The equilibrium $P_2$ of system \eqref{sis} undergoes a supercritical Hopf 
	bifurcation at $b_0=(c-\delta)/(c+\delta)$. For $b>b_0$ the system has a unique stable limit cycle bifurcating from the equilibrium point $P_2$. 
\end{theorem}

\begin{proof}
	The Jacobian matrix at this equilibrium is 
	\begin{equation*}
		A(b)=
		\left(
		\begin{array}{cc}
			-\dfrac{b \delta (c(b-1)+\delta(b+1))}{(c-\delta)^2}    & -\dfrac{b \delta}{c-\delta} \\[0.4cm]
			-\dfrac{b c (b \delta + \delta - c)}{c-\delta}         & 0 
		\end{array}
		\right),
	\end{equation*}
	and it has eigenvalues $\mu(b)\pm \omega(b) i$, where 
	\begin{equation}
		\mu(b)= \frac{b}{2(c-\delta)^2} A \;\;\; \text{and} \;\;\;  \omega(b)=\frac{b}{2(c-\delta)^2}\sqrt{-\delta B}.
	\end{equation}
	We get $\mu(b_0)=0$ for 
	\begin{equation}
		b_0=\dfrac{c-\delta}{c+\delta}.
	\end{equation}
	
	We are working under condition $B<0$ and from this condition it can be deduced that $c-\delta>0$, so the expression of $b_0$ obtained is positive. Therefore at $b=b_0$ the equilibrium point $P_2$ has a pair of pure imaginary eigenvalues $\pm i \omega (b)$ and the system will have a Hopf bifurcation if some Lyapunov constant is nonzero and $(d\mu/db)(b_0)\ne 0$.
	
	The equilibrium is stable for $b>b_0$ (i.e. for $A<0$) and unstable for $b<b_0$ (i.e. for $A>0$). In order to analyze this Hopf bifurcation we will apply  \cite[Theorem 3.3]{Kuznetsov}, so we must prove if the genericity conditions are satisfied. We check that the transversality condition is satisfied as
	\begin{equation}
		\frac{d\mu}{db}(b_0)=- \dfrac{\delta}{2(c-\delta)}<0,
	\end{equation}
	and the sign is determined because $c-\delta>0$.
	
	To check the second condition we must compute the first Lyapunov constant. We fix the value $b=b_0$ and then the equilibrium $P_2$ has the expression
	\begin{equation}
		P_2=\left( \frac{\delta}{c+\delta}, \frac{c^2}{(c+\delta)^2}\right).
	\end{equation}
	We translate $P_2$ to the origin of coordinates obtainig the system
	\begin{equation}
		\begin{split}
			\dot{\varepsilon}_1&= - \varepsilon_1^3  - \dfrac{\delta}{c+\delta} \varepsilon_1^2 - \varepsilon_1\varepsilon_2 - \dfrac{\delta}{c+\delta} \varepsilon_2,\\[0.2cm]
			\dot{\varepsilon_2}&= (c-\delta) \varepsilon_1 \varepsilon_2  + \dfrac{c^2 (c-\delta)}{(c+\delta)^2},\\
		\end{split}
	\end{equation}
	which can be represented as
	\begin{equation}
		\dot{\varepsilon}=A\varepsilon + \frac{1}{2}B(\varepsilon,\varepsilon)+\frac{1}{6}C(\varepsilon,\varepsilon,\varepsilon),
	\end{equation}
	where $A=A(b_0)$ and the multilinear functions $B$ and $C$ are given by
	\begin{equation*}
		B(\varepsilon,\eta)=
		\left(
		\begin{array}{c}
			- \dfrac{2\delta}{c+\delta}\: \varepsilon_1 \eta_1 - \varepsilon_1 \eta_2 - \varepsilon_2 \eta_1\\[0.4cm]
			(c-\delta) \varepsilon_1 \eta_2 + (c-\delta) \varepsilon_2 \eta_1
		\end{array}
		\right),
	\end{equation*}
	\begin{equation*}
		C(\varepsilon,\eta,\zeta)=
		\left(
		\begin{array}{c}
			\- 6 \varepsilon_1 \eta_1 \zeta_1 \\[0.4cm]
			0
		\end{array}
		\right).
	\end{equation*}
	
	We need to find two eigenvectors $p,q$ of the matrix $A$ verifying
	\begin{equation*}
		Aq=i\omega q, \hspace{1cm} A^Tp=-i\omega p,  \hspace{0.4cm} \text{and} \hspace{0.4cm} <p,q>=1,
	\end{equation*}
	as for example 
	\begin{equation}
		q= 
		\left(
		\begin{array}{c}
			-\dfrac{\delta}{c+\delta}	 \\[0.4cm]
			i \omega
		\end{array}
		\right) \;\; \text{and} \;\;
		p=
		\left(
		\begin{array}{c}
			- \dfrac{c+\delta}{2\delta} \\[0.4cm]
			i \omega \dfrac{(c+\delta)^3}{2c^2\delta (c-\delta)}
		\end{array}
		\right).
	\end{equation}
	
	Now we compute 
	\begin{equation*}
		g_{20}=\left\langle p,B(q,q)\right\rangle = \dfrac{\omega^2 (c+\delta)^5- c^2\delta^2 (c+\delta)}{2 \delta c^4 (c-\delta)} + \dfrac{\omega (c+\delta)^3}{2 c^2 \delta (c-\delta)}\: i,
	\end{equation*}
	\begin{equation*}
		g_{11}=\left\langle p,B(q,\overline{q})\right\rangle = - \dfrac{\delta (c+\delta)}{2 c^2 (c-\delta)}, \;\;\; 	g_{21}=\left\langle p,C(q,q,\overline{q})\right\rangle = - \dfrac{3 (c+\delta)^4}{4 c^4 (c-\delta)^2},
	\end{equation*}
	and the first Lyapynov coefficient
	\begin{equation*}
		\ell_1=\frac{1}{2\omega^2} Re(i g_{20}g_{11}+\omega g_{21})= 
		-\dfrac{(c+\delta)^4}{4 c^4 \omega (c-\delta)^2},
	\end{equation*}
	which is negative for any values of the parameters, and so the second condition of 
	the theorem we are applying is satisfied and we can conclude that a unique stable limit cycle bifurcates from the equilibrium point $P_2$ through a Hopf Bifurcation for $b<b_0$ with $b_0-b$ sufficiently small.
\end{proof}

\begin{proposition}
	If $0<b\delta<c-\delta$ and $A>0$, the limit cycle surrounding singular point $P_2$ is unique.
\end{proposition}
\begin{proof}
	This result follows from \cite{LiouCheng} by proving that system \eqref{sis} with  $0<b\delta<c-\delta$ and $A>0$ satisfies conditions (i)-(iv) in Section 2 of \cite{LiouCheng}.
	
	Condition (i) holds taking $g(x)=(c-\delta)x$ which verifies $g(0)=0$ and $g'(x)>0$ for all $x\geq0$ as we have assumed $c-\delta>0$. 
	
	Condition (ii) holds for $f(x)=-x^2+(1-b)x+b$, $K=1$ and $a=(1-b)/2$. From condition $A>0$ we deduce that
	\begin{equation*}
		\delta(c-\delta)-b\delta(c+\delta)>0 \Rightarrow \dfrac{c-\delta}{c+\delta}>\dfrac{b\delta}{\delta} \Rightarrow 1> \dfrac{c-\delta}{c+\delta}>b,
	\end{equation*}
	and condition $b<1$ guarantees that $a>0$. 
	
	Condition (iii) holds for $\lambda=b\delta$ and $x^*= \delta b/(c-\delta)$.  It can be proved that with the expressions chosen for $a$ and $x^*$ the condition $x^*<a$, is equivalent to the condition $A>0$: 
	\begin{equation*}
		x^*< a \Leftrightarrow \dfrac{\delta b}{c-\delta}<\dfrac{1-b}{2} \Leftrightarrow 2\delta b <(1-b)(c-\delta) \Leftrightarrow \delta b + b c < c - \delta \Leftrightarrow b<\dfrac{c-\delta}{c+\delta} \Leftrightarrow A>0.
	\end{equation*}
	
	Condition (iv) is satisfied with
	\begin{equation*}
		x^*=\dfrac{\delta b }{c-\delta} \;\;\; \text{and} \;\;\; \overline{x}^*=1-\dfrac{b c }{c-\delta}.
	\end{equation*}
	We have 
	\begin{equation}
		\dfrac{d}{dx} \dfrac{x f'(x)}{g(x)-\lambda} = \dfrac{-2 x^2 (c-\delta) + 4 x \delta b (b-1)}{((c-\delta)x-\delta b)^2}, 
	\end{equation}
	which is always negative as the polynomial in the numerator is negative in $x=0$ and has no real roots. 
	
	Then, as conditions (i)-(iv) hold for our systems, we can conclude that the limit cycle is unique.
\end{proof}

\begin{remark}
	Theorem $\ref{th_HopfBif}$ proves that the unique limit cycle of system \eqref{sis} appears from the equilibrium point $P_2$ in a Hopf bifurcation. From the proof of Theorem $\ref{th_HopfBif}$ the singular point $P_2$ when $B<0$ and $A=0$ is a weak stable focus.
\end{remark}

So far we have not proved if in cases 4, 6, and 7 of Table \ref{tab:cases} there are or not limit cycles. The following result proves that in some subcases there are not limit cycles. 

\begin{theorem}\label{thm:noperiodic}
	If $0<b \delta < c - \delta$, $A<0$ and
	$1+c < \delta + b + b\delta$,
	then system \eqref{sis}  does not have periodic orbits in the set $\{(x,y)\in\R^2: 
	x,z\geq0\}$. 
\end{theorem}

\begin{proof}
	Let 
	\begin{equation*}
		f(x,y)=x(-x^2+(1-b)x-y+b)\;\; \text{and} \;\; g(x,y)=  y((c-\delta)x - \delta b).
	\end{equation*}
	
	In order to prove the non existence of periodic orbits  we  use 
	Bendixson-Dulac Theorem that states that if there exists a function $\varphi(x,y)$
	such that the term
	\begin{equation*}
		\Delta(x,y)=\frac{\partial (\varphi f)}{\partial x}+\frac {\partial (\varphi g)}{\partial y}
	\end{equation*}
	does not change sign in a simply connected set  $\mathcal{S}$, then there are no 
	periodic orbits on  $\mathcal{S}$.\\
	We consider the function $\varphi(x,y)=1/x$, then:
	
	\begin{equation*}
		\Delta(x,y)= 1 + c - \delta - 2x - \dfrac{b (\delta+x)}{x}.
	\end{equation*}
	We observe that there are no periodic orbits in the set 
	\begin{equation*}
		\{(x,y)\in\R_+^2: \; x\geq 1\},
	\end{equation*}
	because $\dot{x}<0$ for all the points in this set and for the same reason there are no periodic orbits crossing the line
	$\{x=1, y\geq0\}$. As a consequence we can restrict to the case $x< 1$
	for which we obtain
	
	\begin{equation*}
		\Delta(x,y)< 1 + c - \delta - \dfrac{b\delta}{x} - b < 1 + c - \delta - b\delta - b.
	\end{equation*}
	Then $\Delta(x,y)<0$  in $\left\lbrace (x,y)\in\R^2:\; 0 \leq x\leq 1, y \geq 0 \right\rbrace $
	if $1+c-\delta-b-b\delta<0$ and we conclude that there are no
	periodic orbits in the whole set $\left\lbrace (x,y)\in\R^2: x\geq 0, y \geq 0\right\rbrace $.
	
\end{proof}

\begin{conjecture*}
	If $0<b \delta < c - \delta$, $A<0$ (i. e, we are in cases 4,6, or 7 of  Table \ref{tab:cases}) and $1+c > \delta + b + b\delta$,  there are not limit cycles.
\end{conjecture*}

We have numerical evidences that the conjecture holds.

\subsection{Phase portraits on the positive quadrant of the Poincaré disc}

Now we study the global phase portraits of system \eqref{sis} on the positive quadrant of the Poincaré disc when there is a singular point in the positive quadrant, assumming the previous Conjecture.
\color{black}

In case 3 of Table \ref{tab:cases},  by Theorem \ref{th_LimCic3-5} there exist a unique  limit cycle which is the $\omega$-limit of all orbits leaving $O_1$ and $P_2$, and also the $\omega$-limit of the unstable separatrix leaving $P_1$ in the positive quadrant. Then the global phase portraits is the one on Figure \ref{fig:globales}(B).

In case 5 of Table \ref{tab:cases} we have  again that there exists a unique limit cycle attracting all orbits in the positive quadrant. The global phase portrait is the same as the one in case 3 but here the singular point in the positive quadrant is an unstable focus instead of an unstable node. As the local phase portraits of these two singular points are topologically equivalent we have again phase portrait (B) of Figure \ref{fig:globales}.

In cases 4, 6 and 7 of Table \ref{tab:cases}, if $1+c<\delta+b+b\delta$ we have proved that there are no limit cycles.
In case 4 the only possibility is that the stable node $P_2$ is a global attractor for all orbits in the positive quadrant, and we have the global phase portrait given in Figure \ref{fig:globales}(C). In cases 6 and 7 of Table \ref{tab:cases}, $P_2$ is a stable focus and attracts all the orbits of the positive quadrant. As the local phase portrait of a stable focus is topologically equivalent to a stable node,  we also have here the phase portrait of Figure \ref{fig:globales}(C).

In the cases 4, 6 and 7 of Table \ref{tab:cases}, if the conditions $1+c<\delta+b+b\delta$ does not hold, we have asummed that there are not limit cycles, so the conjectured phase portraits will be the same.

\section*{Acknowledgements}

The first and third authors are partially supported by the Ministerio de Ciencia e Innovación, Agencia Estatal de Investigación (Spain), grant PID2020-115155GB-I00 and the Consellería de Educación, Universidade e Formación Profesional (Xunta de Galicia), grant ED431C 2019/10 with FEDER funds. The first author is also supported by the Ministerio de Educacion, Cultura y Deporte de España, contract FPU17/02125.

The second author is partially supported by the Agencia Estatal de Investigaci\'on grant PID2019-104658GB-I00, and the H2020 European Research Council grant MSCA-RISE-2017-777911.


\end{document}